\documentclass[11pt,hyp,]{nyjm}
\usepackage{hyperref}
\hypersetup{nesting=true,debug=true,naturalnames=true}
\usepackage{graphicx,amssymb,upref}

\usepackage[capitalize]{cleveref}

\title{Isomorphisms of Cayley graphs on nilpotent groups}

\author{Dave Witte Morris}  
\address{Department of Mathematics and Computer Science, University of Lethbridge,
Lethbridge, Alberta, T1K~3M4, Canada} 
\email{dave.morris@uleth.ca}  

\author{Joy Morris}  
\address{Department of Mathematics and Computer Science, University of Lethbridge,
Lethbridge, Alberta, T1K~3M4, Canada} 
\email{joy.morris@uleth.ca}  

\author{Gabriel Verret}  
\address{Centre for the Mathematics of Symmetry and Computation,
The University of Western Australia,
35 Stirling Highway, Crawley, Western Australia, 6009, Australia\newline
Also affiliated with FAMNIT, University of Primorska, Glagolja\v{s}ka 8, SI-6000 Koper, Slovenia.}
\email{gabriel.verret@uwa.edu.au}  

\keywords{Cayley graph, nilpotent group, isomorphism, Cayley isomorphism property}

\subjclass[2010]{05C25, 20F18, 05C63, 20F65}

\makeatletter
\newcommand{\noprelistbreak}{\@nobreaktrue\nopagebreak} 
\makeatother

\numberwithin{equation}{section}

\newtheorem{lem}[equation]{Lemma}
\newtheorem{prop}[equation]{Proposition}
\newtheorem{thm}[equation]{Theorem}
\newtheorem{cor}[equation]{Corollary}
\newtheorem{Step}{Step}

\Crefname{thm}{Theorem}{Theorems}
\Crefname{prop}{Proposition}{Propositions}
\Crefname{lem}{Lemma}{Lemmas}

\crefmultiformat{thm}{Theorems~#2#1#3}{ and~#2#1#3}{, #2#1#3}{, and~#2#1#3}

\theoremstyle{definition}
\newtheorem*{notation*}{Notation}
\newtheorem{eg}[equation]{Example}
\newtheorem{defn}[equation]{Definition}
\newtheorem{rem}[equation]{Remark}
\newtheorem{rems}[equation]{Remarks}
\newtheorem*{ack}{Acknowledgments}

\Crefname{rems}{Remark}{Remarks}

\newcommand{\pref}[1]{(\ref{#1})}

\newcommand{\fullccf}[2]{\textup(cf.\ \fullcref{#1}{#2}\textup)}
\newcommand{\fullcref}[2]{\cref{#1}\pref{#1-#2}}
\newcommand{\fullCref}[2]{\Cref{#1}\pref{#1-#2}}
\newcommand{\fullcsee}[2]{(see \fullcref{#1}{#2})}
 
\DeclareMathOperator{\Aut}{Aut}
\DeclareMathOperator{\Cay}{Cay}
\DeclareMathOperator{\dist}{dist}
\DeclareMathOperator{\rank}{rank}

\newcommand{\edge}{\mathbin{\hbox{\vrule height 2 pt depth -1.25 pt width 10pt}}}
\newcommand{\iso}{\cong}
\newcommand{\normal}{\triangleleft}
\newcommand{\normaleq}{\trianglelefteq}
\newcommand{\quot}{\overline}

\newcommand{\NN}{\mathbb{N}}
\newcommand{\ZZ}{\mathbb{Z}}

\newcommand\bigset[2]{\left\{\, #1 
 \mathrel{\left| \vphantom {\left\{ #1 \mid #2 \right\} }
 \right.} #2 \,\right\} }

\usepackage{color}
\newcommand{\refnote}[1]{\marginpar{%
	\color{blue}
	\vbox to 0pt{\vss
	$\begin{pmatrix} \text{see} \\[-3pt] \text{note} \\[-3pt] \text{\ref{#1}} \end{pmatrix}$%
	\vskip -1.1\baselineskip}}}
\theoremstyle{definition}
\newtheorem{aid}{}
\numberwithin{aid}{section}
\newcommand{\oldendaid}{}
\let\oldendaid=\endaid
\renewcommand{\endaid}{\oldendaid\bigskip\hrule width\textwidth \filbreak\bigskip}
\crefformat{aid}{Note~#2#1#3}

\begin{document} 
 
\begin{abstract}  
Let $S$ be a finite generating set of a torsion-free, nilpotent group~$G$. We show that every automorphism of the Cayley graph $\Cay(G;S)$ is affine. (That is, every automorphism of the graph is obtained by composing a group automorphism with multiplication by an element of the group.) More generally, we show that if $\Cay(G_1;S_1)$ and $\Cay(G_2;S_2)$ are connected Cayley graphs of finite valency on two nilpotent groups $G_1$ and~$G_2$, then every isomorphism from $\Cay(G_1;S_1)$ to $\Cay(G_2;S_2)$ factors through to a well-defined affine map from $G_1/N_1$ to $G_2/N_2$, where $N_i$~is the torsion subgroup of~$G_i$. For the special case where the groups are abelian, these results were previously proved by A.\,A.\,Ryabchenko and C.\,L\"oh, respectively.
\end{abstract} 
\maketitle
\tableofcontents

\section{Introduction}

It is easy to construct examples of non-isomorphic groups that have isomorphic Cayley graphs,  \refnote{NonIsoGrps}
even if the Cayley graphs are required to be connected and have finite valency. We show that this is not possible when the groups are torsion-free and nilpotent:

\begin{thm} \label{G1isoG2}
Suppose $G_1$ and $G_2$ are torsion-free, nilpotent groups.  If $G_1$ has a connected Cayley graph of finite valency that is isomorphic to a Cayley graph on~$G_2$, then $G_1 \iso G_2$.
\end{thm}

In fact, the next theorem establishes the stronger conclusion\refnote{G1isoG2Aid} that every isomorphism of the Cayley graphs is obtained from an isomorphism of the groups.  

\begin{defn}
Suppose $\varphi \colon G_1 \to G_2$, where $G_1$ and~$G_2$ are groups. We say that $\varphi$ is an \emph{affine bijection} if it is the composition of a group isomorphism and a translation. That is, there exist a group isomorphism $\alpha \colon G_1 \to G_2$ and $h \in G_2$, such that $\varphi(x) = h \cdot \alpha(x)$, for all $x \in G_1$.
\end{defn}

\begin{thm} \label{NilpotentIso}
Assume
	\begin{itemize}
	\item $G_1$ and~$G_2$ are torsion-free, nilpotent groups,
	and
	\item $S_i$ is a finite, symmetric generating set of~$G_i$, for $i = 1,2$.
	\end{itemize}
Then every isomorphism from $\Cay(G_1;S_1)$ to $\Cay(G_2;S_2)$ is an affine bijection.
\end{thm}

\begin{rem}
In the special case where $G_1$ and~$G_2$ are abelian, \cref{NilpotentIso} was proved by A.\,A.\,Ryabchenko \cite{Ryabchenko-CIAbelian}.
\end{rem}

\begin{defn} \label{CIDefn}{\cite[\S6.4]{Li-CISurvey}}
Let $G$ be a group.
	A Cayley graph $\Cay(G;S)$ is said to be \emph{normal} if the left-regular representation of~$G$ \refnote{regrep}
	is a normal subgroup of $\Aut \bigl( \Cay(G;S) \bigr)$ or, equivalently \cite[Lem.~2.2(b)]{Godsil-FullAut}, if every automorphism of $\Cay(G;S)$ is an affine bijection.

	
\end{defn}

\begin{rem}
It is easy to see that the left-regular representation of~$G$ is a subgroup of the automorphism group of every Cayley graph on~$G$. \Cref{CIDefn} 
requires this subgroup to be normal.
\end{rem} 

With this terminology, the special case of \cref{NilpotentIso} in which $G_1 = G_2$ has the following known result as an
immediate consequence.

\begin{cor}[M\"oller-Seifter {\cite[Thm.~4.1(1)]{MoellerSeifter}}] \label{NilpotentIsNormal}
If $G$ is a torsion-free, nilpotent group, then every connected Cayley graph of finite valency on~$G$ is normal. 
\end{cor}



In the statement of \cref{NilpotentIso}, the word ``nilpotent'' cannot be replaced with ``solvable'' (or even ``polycyclic''):

\begin{eg} \label{KleinBottleEg}
Let $G$ be the unique nonabelian semidirect product of the form $\ZZ \rtimes \ZZ$. More precisely, 
	$$G = \langle\, a, b \mid b^{-1} a b = a^{-1}\,\rangle = \langle a \rangle \rtimes \langle b \rangle .$$
(In other words, $G$~is the fundamental group of the Klein bottle.) Then $G$~is obviously polycyclic (so it is solvable), but it is not difficult to see that $\Cay \bigl( G; \{a^{\pm1},b^{\pm1}\} \bigr)$ is not normal. 
(Namely, the map $\varphi(a^i b^j) = b^i a^j$ is a graph automorphism that is not an affine bijection.)
\refnote{KleinBottlePf}
\end{eg}

If $G$ is not torsion-free, then the conclusion of \cref{NilpotentIsNormal} fails:

\begin{prop} \label{TorsionNotNormal}
Let $G$ be a finitely generated, infinite group.
	If $G$ is not torsion-free, then $G$~has a connected Cayley graph of finite valency that is not normal.
\end{prop}

%

However, the next theorem shows that if the torsion-free hypothesis is removed from \cref{NilpotentIso}, then the conclusion still holds modulo the elements of finite order.

\begin{defn}[{}{\cite[1.2.13, p.~11]{LennoxRobinson}}]
Suppose $G$ is a finitely generated, nilpotent group. The set of all elements of finite order in~$G$ is called the \emph{torsion subgroup} of~$G$. (This is a finite, normal subgroup of~$G$.)
\end{defn}

\begin{thm} \label{NilpotentTorsion}
Assume
	\begin{itemize}
	\item $S_i$ is a symmetric, finite generating set of the nilpotent group~$G_i$, for $i = 1,2$,
	\item $\varphi$ is an isomorphism from $\Cay(G_1;S_1)$ to $\Cay(G_2;S_2)$,
	and
	\item $N_i$ is the torsion subgroup of~$G_i$, for $i = 1,2$.
	\end{itemize}
Then $\varphi$ induces a well-defined affine bijection $\quot\varphi \colon G_1/N_1 \to G_2/N_2$.
\end{thm}

\begin{cor} \label{NilpotentTorsionHaveIso}
For $i = 1,2$, assume $N_i$ is the torsion subgroup of the finitely generated, nilpotent group~$G_i$. Then there is a connected Cayley graph of finite valency on~$G_1$ that is isomorphic to a Cayley graph on~$G_2$ if and only if $G_1/N_1 \iso G_2/N_2$ and\/ $|N_1| = |N_2|$.
\end{cor}

\begin{cor} \label{UniqueSharp}
If $\Cay(G;S)$ is any Cayley graph of finite valency on a torsion-free, nilpotent group~$G$, then the left-regular representation of~$G$ is the only nilpotent subgroup of $\Aut \bigl( \Cay(G;S) \bigr)$ that acts sharply transitively on the vertices of the Cayley graph.
\end{cor}

\begin{rems} \label{TorsionRems} \ 
\noprelistbreak
	\begin{enumerate} 
	\item In the special case where $G_1$ and~$G_2$ are abelian, \cref{NilpotentTorsion,NilpotentTorsionHaveIso} were proved by C.\,L\"oh \cite{Loeh-WhichAbel}.
	\item \Cref{NilpotentIso} is the special case of \cref{NilpotentTorsion} in which the torsion subgroups $N_1$ and~$N_2$ are trivial.
	\item Although \cref{NilpotentIso,NilpotentTorsion} are stated only for graphs, they obviously remain true in the setting of Cayley digraphs. This is because any isomorphism of digraphs is also an isomorphism of the underlying graphs.
	\item \label{TorsionRems-NonisoGrps}
	Some non-nilpotent groups have some Cayley graphs that are isomorphic to Cayley graphs on nilpotent groups---or even abelian groups. (For example, the Cayley graph in \cref{KleinBottleEg} is isomorphic to $\Cay \bigl( \ZZ \times \ZZ, \{(\pm1,0), (0,\pm1)\} \bigr)$.) \Cref{NilpotentTorsion} implies that any such group must have a subgroup of finite index that is nilpotent, but this fact is well known to be a consequence of Gromov's famous theorem that groups of polynomial growth are virtually nilpotent \cite{Gromov-PolyGrowth}. Indeed, in order to conclude from Gromov's Theorem that $G$ has a nilpotent subgroup of finite index, it suffices to know that $G$ has a connected Cayley graph of finite valency that is \emph{quasi-isometric} (not necessarily isomorphic) to a Cayley graph on a nilpotent group. 
	\end{enumerate}
\end{rems}

\Cref{NilpotentIso} is proved in \cref{NilpotentSect}, and this result is used to prove \cref{NilpotentTorsion} (and its corollaries) in \cref{NilpotentTorsionSect}.
(The arguments are based on techniques of A.\,A.\,Ryabchenko \cite{Ryabchenko-CIAbelian} and C.\,L\"oh \cite{Loeh-WhichAbel}.)
\Cref{TorsionNotNormal} is proved in \cref{OtherTorsionSect}.

\begin{ack}
This work was partially supported by Australian Research Council grant DE130101001 and a research grant from the Natural Sciences and Engineering Research Council of Canada.
\end{ack}

\section{Preliminaries} \label{PrelimSect}

The following result is the special case of \cref{NilpotentIso} in which $G_1$ and~$G_2$ are abelian. (Although not stated in exactly this form in \cite{Ryabchenko-CIAbelian}, the result follows from the proof that is given there and is reproduced in \cite[Thm.~5.3]{Morris-CIInfinite-Arxiv}). This case is not covered by the proof in \cref{NilpotentSect}.

\begin{prop}[Ryabchenko {\cite[Thm.~2]{Ryabchenko-CIAbelian}}] \label{AbelianCase}
Assume
	\begin{itemize}
	\item $G_1$ and~$G_2$ are torsion-free, abelian groups,
	\item $S_i$ is a symmetric, finite generating set of~$G_i$, for $i = 1,2$,
	and
	\item $\varphi$ is an isomorphism from $\Cay(G_1;S_1)$ to $\Cay(G_2;S_2)$.
	\end{itemize}
Then $\varphi$ is an affine bijection. \refnote{RyabchenkoPf}
\end{prop}

As in \cite{Loeh-WhichAbel}, we use geometric terminology, such as geodesics and convexity, instead of presenting our arguments in group-theoretic language.

\begin{defn}
Let $S$ be a symmetric, finite generating set of a group~$G$. 
	\begin{itemize}
	
	\item For $g ,h \in G$, the distance from~$g$ to~$h$ in the Cayley graph $\Cay(G;S)$ is denoted $\dist_S(g,h)$.
	
	\item A finite sequence $[g_i]_{i=m}^n$ of elements of~$G$ is a \emph{geodesic segment} from~$g_m$ to~$g_n$ in $\Cay(G;S)$ if $\dist_S(g_i,g_j) = |i - j|$ for $m \le i,j \le n$.
	
	\item A bi-infinite sequence $[g_i]_{i=-\infty}^\infty$ of elements of~$G$ is a \emph{geodesic line} in $\Cay(G;S)$ if $\dist_S(g_i,g_j) = |i - j|$ for all $i,j \in \ZZ$.
	
	\item A geodesic line $[g_i]_{i=-\infty}^\infty$ in $\Cay(G;S)$ is \emph{convex} if $[g_i,g_{i+1}, \ldots,g_j]$ is the only path of length $j - i$ from $g_i$ to~$g_j$, for all $i,j \in \ZZ$ (with $i < j$).
	
	\item A geodesic line $[g_i]_{i=-\infty}^\infty$ in $\Cay(G;S)$ is \emph{homogeneous} if there exists $\varphi \in \Aut \bigl( \Cay(G;S) \bigr)$, such that $\varphi(g_i) = g_{i+1}$ for all~$i$.
	
	\item $\Aut_e \bigl( \Cay(G;S) \bigr) = \bigset{ \varphi \in \Aut \bigl( \Cay(G;S) \bigr) }{ \varphi(e) = e }$.
	
	\item Each oriented edge of $\Cay(G;S)$ has a natural label, which is an element of~$S$. Namely, each edge of the form $g \edge gs$ is labelled~$s$. (Note that the same edge with the opposite orientation is labelled~$s^{-1}$.) Each edge in a geodesic segment (or geodesic line) comes with a natural orientation, and therefore has a label. 		
	\end{itemize}
\end{defn}

\begin{lem} \label{phi(G*)}
For $i = 1,2$, assume
	\begin{itemize}
	\item $S_i$ is a symmetric, finite generating set of a group~$G_i$,
	\item $\varphi_i$ is an isomorphism from $\Cay(G_1;S_1)$ to $\Cay(G_2;S_2)$, such that $\varphi_i(e) = e$,
	\item $g_i \in G_i$, 
	\item $S_i^* = \bigset{ \rho(g_i) }{ \rho \in \Aut_e \bigl( \Cay(G_i; S_i) \bigr) }$,
	and
	\item $G_i^* = \langle S_i^* \rangle$.
	\end{itemize}
If $\varphi_1(g_1) = g_2$, then the restriction of $\varphi_2$ to~$G_1^*$ is an isomorphism from $\Cay \bigl( G_1^*;S_1^* \cup (S_1^*)^{-1} \bigr)$ to $\Cay \bigl( G_2^*;S_2^* \cup (S_2^*)^{-1} \bigr)$.
\end{lem}

\begin{proof}
For convenience, let $A_i = \Aut \bigl( \Cay(G_i; S_i) \bigr)$, $A_i^e = \Aut_e \bigl( \Cay(G_i; S_i) \bigr)$, and $\Gamma_i = \Cay \bigl( G_i;S_i^* \cup (S_i^*)^{-1} \bigr)$.
For $\rho \in A_i$ and $g \in G_i$, define $\rho_g \in A_i^e$ by $\rho_g(x) = \rho(g)^{-1} \, \rho(gx)$. Then, since $S_i^*$ is $A_i^e$-invariant, we have
	$$ \rho(g S_i^*) = \rho(g) \, \rho_g(S_i^*) = \rho(g) S_i^* ,$$
so $\rho$ is an automorphism of~$\Gamma_i$. Since $A_i^e$ is transitive on~$S_i^*$, and the left-regular representation of~$G_i$ is transitive on~$G_i$, this implies that the set of edges of~$\Gamma_i$ is the $A_i$-orbit of the edge $e \edge g_i$. 

Since $\varphi_1$ is a graph isomorphism, it maps the $A_1$-orbit of~$g_1$ to the $A_2$-orbit of $\varphi_1(g_1) = g_2$. So $\varphi_1$ is an isomorphism from~$\Gamma_1$ to~$\Gamma_2$. Since the composition $\varphi_2 \circ \varphi_1^{-1}$ is in~$A_2$, and is therefore an automorphism of~$\Gamma_2$, we conclude that $\varphi_2$ is an isomorphism from~$\Gamma_1$ to~$\Gamma_2$. Since $\Cay \bigl( G_i^*;S_i^* \cup (S_i^*)^{-1} \bigr)$ is the component of~$\Gamma_i$ that contains~$e$, and $\varphi_2(e) = e$, the desired conclusion follows.
\end{proof}

\begin{lem}[{}{\cite[Prop.~2.5(3)]{Loeh-WhichAbel}}] \label{EdgeInCentre}
Let $s \in S$ be the label of some edge of a convex geodesic line in $\Cay(G;S)$. If $s \in Z(G)$, then every edge of the geodesic line is labelled~$s$.
\end{lem}

\begin{proof}
Suppose $g_i \edge g_{i+1}$ is labelled~$s$. Let $t$ be the label of $g_{i+1} \edge g_{i+2}$. Then 
	$g_{i+2} = g_i st = (g_i t)s$,
so $[g_i, g_i t, g_{i+2}]$ is a path of length~$2$ from $g_i$ to~$g_{i+2}$. Therefore, convexity implies $[g_i, g_i t, g_{i+2}] = [g_i, g_{i+1},g_{i+2}]$, so $g_i t = g_{i+1} = g_i s$, so $t = s$. This means the label of $g_{i+1} \edge g_{i+2}$ is~$s$. By induction, we see that every edge is labelled~$s$.
\end{proof}

In the remainder of this \lcnamecref{PrelimSect}, we recall some basic facts about nilpotent groups.

\begin{defn}[{}{\cite[p.~38]{LennoxRobinson} or \cite[Notn.~3.4]{Conner-TransNumbs}}]
For a subgroup~$H$ of a group~$G$, we let
	$$ \sqrt{H} = \{\, g \in G \mid \text{$g^k \in H$ for some $k \in \ZZ^+$} \,\} .$$
This is called the \emph{isolator} of~$H$ in~$G$.
\end{defn}

Any finitely generated, abelian group~$A$ is isomorphic to $\ZZ^r \times F$, for some $r \in \ZZ^{\ge 0}$ and finite, abelian group~$F$. The number~$r$ is called the \emph{rank} of~$A$, and is denoted $\rank A$. The following definition generalizes this notion from abelian groups to nilpotent groups.

\begin{defn}[{}{\cite[1.3.3 and p.~85 (1)]{LennoxRobinson}}]
Assume $G$ is a nilpotent group. Then $G$ is solvable, which means there is a series  
	\begin{align*}
	 \{e\} = G_0 \normal G_1 \normal \cdots \normal G_{r-1} \normal G_r = G
	,  \end{align*}
of subgroups of~$G$, such that each quotient $G_i/G_{i-1}$ is abelian. If $G$~is finitely generated, then the \emph{Hirsch rank} of~$G$ is the sum of the ranks of these (finitely generated) abelian groups. That is,
	$$ \rank G = \sum_{i = 1}^r \rank (G_i/G_{i-1})  .$$
It is not difficult to see that this is independent of the choice of the subgroups $G_1,\ldots,G_{r-1}$.
\end{defn}

\begin{lem} \label{NilpLems}
Assume $G$ is a finitely generated, nilpotent group, 
$H$~is a subgroup of~$G$,
and $S$~is a symmetric, finite generating set of~$G$. Then:
\noprelistbreak
	\begin{enumerate}
	
	\item \label{NilpLems-fg}
	\cite[1.2.16, p.~11]{LennoxRobinson} 
	$H$ is finitely generated.

	\item \label{NilpLems-FiniteIndex}
	\cite[2.3.1(ii), p.~39]{LennoxRobinson} 
	$\sqrt{H}$ is a subgroup of~$G$ that contains~$H$, and $\bigl| \sqrt{H} : H \bigr| < \infty$.
	
	\item \label{NilpLems-G/Isolated}
	If $N \normaleq G$, then $\sqrt{N} \normaleq G$ and $G/\sqrt{N}$ is torsion-free. \refnote{IsolateNormal}
	
	\item \label{NilpLems-Zisolated}
	\cite[2.3.8(ii), p.~42]{LennoxRobinson} 
	If $G$ is torsion-free, then $\sqrt{Z(G)} = Z(G)$.
	
	\item \label{NilpLems-CommutatorIsolated}
	\cite[2.3.9(iv), p.~43]{LennoxRobinson} 
	$\bigl[ \sqrt{H}, \sqrt{H} \bigr] \subseteq \sqrt{[H,H]}$.
	
	\item \label{NilpLems-RankQuotient}
	If $N$ is a normal subgroup of~$G$, then $\rank G = \rank N + \rank(G/N)$. Therefore, $\rank(G/N) \le \rank G$, with equality if and only if $N$~is finite.
	
	\item \label{NilpLems-rank}
	\textup(cf.\ \cite[Lem.~2.6, p.~9]{Hall-NilpGrps}\textup) 
	We have $\rank H \le \rank G$, with equality if and only if\/ $|G:H| < \infty$.\refnote{rankH<rankG}
	
	\item \label{NilpLems-MeetCentre}
	\cite[5.2.1, p.~129]{Robinson-CourseGroups}
	If $N$ is a nontrivial normal subgroup of~$G$, then $N \cap Z(G)$ is nontrivial.
	
	\item \label{NilpLems-FiniteCjgs}
	\textup(cf.\ \cite[2.3.8(i), p.~42]{LennoxRobinson}\textup) \refnote{FinitelyManyConjugates}
	If $G$ is torsion-free, then the elements of $Z(G)$ are the only elements of~$G$ that have only finitely many conjugates.
	
	\item \label{NilpLems-BddDist}
	\textup(cf.\ \cite[2.1.2, p.~30]{LennoxRobinson}\textup) Assume $G$ is torsion-free, and $a,b,g \in G$. If
		$$\sup_{k \in \ZZ^+} \dist_S(a^k, gb^k) < \infty ,$$
	then $b = g^{-1} a g$.\refnote{BddDist}
	
	\item \label{NilpLems-distorted}
	\cite[Lem.~3.5(i,iii)]{Conner-TransNumbs}
	For $g \in G$, we have $g \in \sqrt{[G,G]}$ if and only if\/ $\dist_S(e,g^k)/k \rightarrow 0$ as $k \to \infty$. \refnote{DistortedNote}

	\end{enumerate}
\end{lem}

\begin{rem}
\fullCref{NilpLems}{CommutatorIsolated} corrects a typographical error. It is stated in \cite[2.3.9(iv), p.~43]{LennoxRobinson} that equality holds, but a counterexample to this is provided by any finite-index subgroup~$G$ of the discrete Heisenberg group, such that $[G,G]$ is a proper subgroup of $Z(G)$: letting $H = G$, we have
	$$ \bigl[ \sqrt{G}, \sqrt{G} \bigr] = [G,G] \neq Z(G) = \sqrt{[G,G]} .$$
\end{rem}

\begin{defn}
A group~$G$ is \emph{bi-orderable} if it is has a total order~$\prec$ that is invariant under both left-translations and right-translations. (That is, $x \prec y \Rightarrow axb \prec ayb$ for all $x,y,a,b \in G$.)
\end{defn}

\begin{lem}[{}{\cite[Cor.~3.3.2, p.~57]{KopytovMedvedev}}] \label{NilpIsBO}
Every torsion-free, nilpotent group is bi-orderable.
\end{lem}

\begin{lem}[cf.\ {\cite[1st paragraph of \S4]{Ryabchenko-CIAbelian} or \cite[Prop.~2.9(1)]{Loeh-WhichAbel}}] \label{ConvexGeodesic}
If $S$ is a finite generating set of a nontrivial, bi-orderable group~$G$, then there exists $s \in S$, such that $[s^i]_{i=-\infty}^\infty$ is a convex geodesic line in $\Cay(G;S \cup S^{-1})$.
\end{lem}

\begin{proof}
Let $\prec$ be a total order on~$G$ that is invariant under both left-translations and right-translations. Since the set $S \cup S^{-1}$ is finite, it has a maximal element~$s$ under this order. We may assume $s \in S$, \refnote{InverseOrder}
by replacing $\prec$ with the order $\prec'$ defined by $x \prec' y \Leftrightarrow x^{-1} \prec y^{-1}$, if necessary.

For $a,b,c,d \in G$ with $a \preceq b$ and $c \preceq d$, the invariance under translations implies that $ac \preceq bd$ (and equality holds iff $a = b$ and $c = d$). \refnote{ac<bd}
By induction on~$k$, we conclude that $s_1 s_2 \cdots s_k \preceq s^k$ for all $s_1, s_2, \ldots, s_k \in S \cup S^{-1}$, and that equality holds iff $s_1 = s_2 = \cdots = s_k = s$. This implies that $[s^i]_{i=-\infty}^\infty$ is a convex geodesic line.
\end{proof}

\section{Torsion-free nilpotent groups} \label{NilpotentSect}

In this section, we prove \cref{NilpotentIso}.
Let $\varphi$ be an isomorphism from $\Cay(G_1;S_1)$ to $\Cay(G_2;S_2)$.
By composing with a left translation, we may assume $\varphi(e) = e$.\refnote{TranslateToIdentity}
(Under this assumption, we will show that $\varphi$ is a group homomorphism. Since $\varphi$ is bijective, it must then be a group isomorphism.)
The proof is by induction on $\rank G_1 + \rank G_2$.

\begin{notation*}
Let $Z_i^\dagger = Z(G_i) \cap \sqrt{[G_i,G_i]}$ for $i = 1,2$. 
\end{notation*}

\begin{Step} \label{ZGeodToGeod}
For every $g \in G_1$ and $z \in Z_1^\dagger$, there exists $\sigma_g(z) \in G_2$, such that $\varphi(gz^k) = \varphi(g) \, \sigma_g(z)^k$ for all $k \in \ZZ$.
\end{Step}

\begin{proof}
By composing with left translations in~$G_1$ and~$G_2$, we may assume $g = e$. \refnote{g=e}
Define $S_1^*$, $S_2^*$,  $G_1^*$, and~$G_2^*$ as in \cref{phi(G*)}, with $g_1 = z$ and $g_2 = \varphi(z)$.
Combining \cref{NilpIsBO,ConvexGeodesic} yields $s \in S_2^*$, such that 
	$$ \text{$[s^i]_{i=-\infty}^\infty$ is a convex geodesic line in $\Cay \bigl( G_2^*;S_2^* \cup (S_2^*)^{-1} \bigr)$.} $$
The definition of~$S_2^*$ implies there is an isomorphism $\psi$ from $\Cay(G_1; S_1)$ to $\Cay(G_2; S_2)$ with $\psi(e) = e$ and $\psi(z) = s$. \refnote{psi(z)=s}
Since \cref{phi(G*)} tells us that $\psi$ restricts to an isomorphism from $\Cay \bigl( G_1^*;S_1^* \cup (S_1^*)^{-1} \bigr)$ to $\Cay \bigl( G_2^*;S_2^* \cup (S_2^*)^{-1} \bigr)$, we know that $\psi^{-1} \bigl( [s^i]_{i=-\infty}^\infty \bigr)$ is a convex geodesic line in $\Cay \bigl( G_1^*;S_1^* \cup (S_1^*)^{-1} \bigr)$. From the choice of~$\psi$, this geodesic line contains the edge $e \edge z$, so \cref{EdgeInCentre} tells us that this geodesic line must be $[z^i]_{i=-\infty}^\infty$. This means $\dist_{S_1^*}(z^i, z^j) = |i - j|$ for all $i,j \in \ZZ$. We conclude from \fullcref{NilpLems}{distorted} that $z \notin \sqrt{[G_1^*, G_1^*]}$.

On the other hand, since $z \in Z_1^\dagger$, we know that $z \in \sqrt{[G_1,G_1]}$. Therefore $\sqrt{[G_1^*,G_1^*]} \neq \sqrt{[G_1,G_1]}$. This implies that $[G_1^*,G_1^*]$ has infinite index in $[G_1,G_1]$  \fullccf{NilpLems}{FiniteIndex}, so $G_1^*$ must have infinite index in~$G_1$ \fullccf{NilpLems}{CommutatorIsolated}. \refnote{CommutatorSmall}
Therefore, $\rank G_1^* + \rank G_2^* < \rank G_1 + \rank G_2$ (see \cref{NilpLems}(\ref{NilpLems-rank})), so our induction hypothesis tells us that the restriction of~$\varphi$ to~$G_1^*$ is a group isomorphism onto~$G_2^*$. Hence, $\varphi(z^k) = \varphi(z)^k$ for all~$k$, so we may let $\sigma_g(z) = \varphi(z)$.
\end{proof}

\begin{Step} \label{ZtoZ}
 We have $\varphi( x Z_1^\dagger) = \varphi(x)Z_2^\dagger$, for all $x \in G_1$.
 \end{Step}
 
 \begin{proof}
By composing with left translations in~$G_1$ and~$G_2$, we may assume $x = e$.
Then, since $\varphi^{-1}$ is also an isomorphism, it suffices to show $\varphi( Z_1^\dagger \bigr) \subseteq Z_2^\dagger$.
 Fix $z \in Z_1^\dagger$. 
 For all $k \in \ZZ$, we have $\dist_{S_1}(z^k, g z^k) = \dist_{S_1}(e, g)$ (because $z \in Z(G_1)$). Since $\varphi$ is a graph isomorphism, this implies 
	$\dist_{S_2} \bigl( \varphi(z)^k, \varphi(g) \, \sigma_g(z)^k \bigr)$ does not depend on~$k$.
So \cref{NilpLems}(\ref{NilpLems-BddDist}) tells us that $\varphi(g)^{-1} \varphi(z) \varphi(g) = \sigma_g(z)$. 
From the definition of $\sigma_g(z)$, we see that $\dist_S \bigl( e, \sigma_g(z) \bigr) = \dist_S(e,z)$, so this implies that $\varphi(g)^{-1} \varphi(z) \varphi(g)$ is in a ball of fixed radius, independent of~$g$. Since $\varphi(g)$ is an arbitrary element of~$G_2$, we conclude that $\varphi(z)$ has only finitely many conjugates. Since $G_2$ is torsion-free nilpotent, this implies $\varphi(z) \in Z(G_2)$ \fullcsee{NilpLems}{FiniteCjgs}.

Also, we see from \fullcref{NilpLems}{distorted} that $\varphi \bigl(\! \sqrt{[G_1,G_1]} \,\bigr) = \sqrt{[G_2,G_2]}$ (since $\varphi$~is a graph isomorphism). Therefore $\varphi(z) \in \sqrt{[G_2,G_2]}$. So $\varphi(z) \in Z_2^\dagger$.
\end{proof}

\begin{Step}
Completion of the proof of \cref{NilpotentIso}.
\end{Step}

\begin{proof}
Let $\quot{G_i} = G_i/Z_i^\dagger$ for $i = 1,2$. Note that $Z_1^\dagger$ is finitely generated \fullcsee{NilpLems}{fg}. Therefore, by passing to a power of the graphs $\Cay(G_1; S_1)$ and $\Cay(G_2; S_2)$ (or, in other words, by replacing $S_i$ with an appropriate product $(S_i \cup \{e\})  (S_i \cup \{e\})  \cdots (S_i \cup \{e\})$), we may assume that $\Cay(Z_1^\dagger; S_1 \cap Z_1^\dagger)$ is connected. \refnote{Zconnected}
From \cref{ZtoZ}, we know that $\varphi$ induces a well-defined isomorphism $\quot\varphi$ from $\Cay(\quot{G_1} ; S_1)$ to $\Cay(\quot{G_2} ; S_2)$. 

We may assume that $G_1$ and~$G_2$ are not both abelian (otherwise, Ryabchenko's Theorem \pref{AbelianCase} applies), so either $[G_1,G_1]$ or $[G_2,G_2]$ is nontrivial. This implies that either $Z_1^\dagger$ or $Z_2^\dagger$ is nontrivial \fullcsee{NilpLems}{MeetCentre}, and therefore infinite (since $G_1$ and~$G_2$ are torsion-free). Hence, we have $\rank \quot{G_1} + \rank \quot{G_2} < \rank G_1 + \rank G_2$ \fullcsee{NilpLems}{RankQuotient}, so, by induction on $\rank G_1 + \rank G_2$, we may assume that $\quot\varphi$ is a group isomorphism from~$\quot{G_1}$ to~$\quot{G_2}$ (since \fullcref{NilpLems}{Zisolated} implies that $\quot{G_1}$ and $\quot{G_2}$ are torsion free).\refnote{G/ZdaggerTorsFree}

For each $g \in G_1$ and $z \in Z_1^\dagger$, we have
	\begin{align*}
	 \dist_{S_2} \bigl( \sigma_e(z)^k, \varphi(g) \, \sigma_g(z)^k \bigr) 
	&= \dist_{S_2} \bigl( \varphi( z^k), \varphi( g z^k) \bigr)
	\\&= \dist_{S_1}(z^k,g z^k)
	\\&= \dist_{S_1}(e,g) 
	, \end{align*}
since $z \in Z(G_1)$. Then, from \fullcref{NilpLems}{BddDist} (and the fact that \cref{ZtoZ} tells us that $\sigma_e(z)$ is in~$Z_2^\dagger$ and therefore commutes with $\varphi(g)$), we see that $\sigma_g(z) = \sigma_ e(z)$. This means $\sigma_g(z)$ is independent of~$g$ (so we may drop the subscript).

Fix some $g \in G_1$ and $s \in S_1$. We have $\varphi(gs) = \varphi(g) \, \varphi(s) \, \sigma(z)$, for some $z \in Z_1^\dagger$ (because $\quot{\varphi}$ is a homomorphism and the surjectivity in \cref{ZtoZ} tells us $\sigma(Z_1^\dagger) = Z_2^\dagger$). Consider any $k \ge 0$ with $sz^k \in S_1$. Then
	$$ \varphi(gs z^k)
	= \varphi(gs) \, \sigma(z)^k
	= \varphi(g) \, \varphi(s) \, \sigma(z) \, \sigma(z)^k
	= \varphi(g) \, \varphi(s z^{k+1}) .$$
Since $\varphi$ is a graph homomorphism and, by assumption, $s z^k \in S_1$, we must have $\varphi(s z^{k+1}) \in S_2$. So $s z^{k+1} \in S_1$. By induction (with $k = 0$ as the base case), we conclude that $sz^k \in S_1$ for all $k \in \ZZ^+$. Since $S_1$ is finite (and $G_1$ is torsion-free), this implies $z = e$. So $\varphi(gs) = \varphi(g) \, \varphi(s)$. Since $g$ is an arbitrary element of~$G_1$ and $s$~is an arbitrary element of the generating set~$S_1$, this implies that $\varphi$ is a group homomorphism.
\end{proof}

\section{Nilpotent groups that may have torsion} \label{NilpotentTorsionSect}

\begin{prop} \label{phi(N)}
Assume
\noprelistbreak
	\begin{itemize}
	\item $S$ is a finite generating set of the group~$G$,
	and
	\item $N$ is a finite, normal subgroup of~$G$, such that $G/N$ is bi-orderable.
	\end{itemize}
Then every automorphism of $\Cay(G;S)$ induces a well-defined automorphism of $\Cay(G/N;S)$.
\end{prop}

\begin{proof}
Let 
	$$N^* = \{\, \varphi(n) \mid \varphi \in \Aut_e \bigl( \Cay(G; S) \bigr), n \in N \,\} .$$
It is important to note that, since $N$ is contained in a ball of finite radius centred at~$e$, and $N^*$~must be contained in that same ball, the set~$N^*$ is finite.
We wish to show $N^* \subseteq N$.

Assume, without loss of generality, that $N \subseteq S$ (by passing to a power of $\Cay(G;S)$).
Since $\langle N^* \rangle$ is obviously invariant under $\Aut_e \bigl( \Cay(G; S) \bigr)$, there is no harm in assuming $\langle N^* \rangle = G$\refnote{Nstar}.

Let $\quot G = G/N$, and let $\quot{N^*}=\{gN \mid g \in N^*\}$. We wish to show $\quot G$ is trivial. Suppose not. (This will lead to a contradiction.) Since, by assumption, $\quot G$ is bi-orderable, \cref{ConvexGeodesic} provides $g \in N^*$, such that $[\quot g^i]_{i=-\infty}^\infty$ is a geodesic line in $\Cay(\quot G; \quot{N^*})$. Then, since the natural map $\Cay(G; N^*) \to \Cay(\quot G; \quot{N^*})$ decreases distances, it is clear that $\gamma = [g^i]_{i=-\infty}^\infty$ is a geodesic line in $\Cay(G; N^*)$. By the definition of~$N^*$, there exists $\varphi \in \Aut_e \bigl( \Cay(G;S) \bigr)$, such that $\varphi(g) \in N$. Then $\varphi(\gamma)$ is a geodesic line that contains the edge $e \edge n$ for some $n \in N$.

To obtain the contradiction that completes the proof, we use an argument of C.\,L\"oh \cite[first paragraph of page~105]{Loeh-WhichAbel}. Write $\varphi(\gamma) = [h_i]_{i=-\infty}^\infty$. 
For each $k \in \NN$, let $\#(k)$ be the number of geodesic segments from $h_i$ to~$h_{i + k}$. (Since $\gamma = [g^i]_{i=-\infty}^\infty$ is obviously homogeneous, we know that $\varphi(\gamma)$ is also homogeneous, so $\#(k)$ is independent of the choice of~$i$.)
We may assume $h_0 = e$ (so $h_1 = n$). 
Since $N$ is a finite normal subgroup of~$G$, it is easy to see that no geodesic segment can contain two edges that are labelled by elements of~$N$. (Namely, if $(n, s_1,\ldots,s_k, n')$ is a path in $\Cay(G;N)$, then there exists $n'' \in N_1$, such that $n'' s_1 \cdots s_k = n s_1 \cdots s_k n'$, so $(n'' ,s_1, \ldots, s_k)$ is a shorter path with the same endpoints.) 
Hence, for all $k > 1$, no geodesic segment from $h_1$ to~$h_k$ has any edges that are labelled by elements of~$N$. (Otherwise, concatenating $(n)$ at the start would yield a geodesic segment from~$h_0$ to $h_k$ with more than one edge labelled by elements of~$N$.)

For any geodesic segment $\gamma' = (s_1,\ldots,s_k)$ from $h_1$ to~$h_{k + 1}$, we can construct $k + 1$ different geodesic segments $\gamma_1,\ldots,\gamma_{k+1}$ from $h_0$ to~$h_{k + 1}$, by inserting a single edge labelled by an element of~$N$, as follows:
	$$ \gamma_i = (s_1,s_2,\ldots,s_{i-1}, n_i, s_i,\ldots,s_k) ,$$
where $n_i \in N$ is chosen so that $n s_1 s_2 \cdots s_{i-1} = s_1 s_2 \cdots s_{i-1} n_i$. (This is possible because the subgroup~$N$ is normal.) This implies $\#(k + 1) \ge (k+1) \cdot \#(k)$, for all~$k$. Therefore $\#(k) \ge k!$\,. However, it is clear that $\#(k) \le |S|^k$, so this contradicts the fact that factorials grow faster than exponentials.
\end{proof}

Combining this \lcnamecref{phi(N)} with \cref{NilpotentIso} yields the following slight generalization of \cref{NilpotentTorsion} that allows $G_1$ and~$G_2$ to be slightly non-nilpotent:

\begin{thm} \label{AlmNilpotent}
Assume
\noprelistbreak
	\begin{itemize}
	\item $S_i$ is a symmetric, finite generating set of the group~$G_i$, for $i = 1,2$,
	\item $N_i$ is a finite, normal subgroup of~$G_i$, such that $G_i/N_i$ is torsion-free nilpotent, for $i = 1,2$,
	and
	\item $\varphi$ is an isomorphism from $\Cay(G_1;S_1)$ to $\Cay(G_2;S_2)$.
	\end{itemize}
Then $\varphi$ induces a well-defined affine bijection $\quot\varphi \colon G_1/N_1 \to G_2/N_2$.
\end{thm}

\begin{proof}
By using $\varphi$ to identify $\Cay(G_1;S_1)$ with $\Cay(G_2;S_2)$, we can realize $G_2$ as a sharply transitive subgroup~$G_2'$ of $\Aut \bigl( \Cay(G_1; S_1) \bigr)$. (Namely, for $h \in G_2$, define $h'(x) = \varphi^{-1} \bigl( h \, \varphi(x) \bigr)$.) 

For any $g \in G_1$ and $n \in N_1$, there exists $h \in G_2$, such that $\varphi(gn) = h \, \varphi(g)$. This means $h'g = gn \in g N_1$. 
From \cref{phi(N)} (and \cref{NilpIsBO}), we know that $G_2'$ factors through to a well-defined group of permutations on $G_1/N_1$, so this implies $h'(gN_1) = g N_1$. Since $gN_1$ is finite (and $G_2'$ is sharply transitive), we conclude that $h'$~has finite order, so $h'$~is in the torsion subgroup~$N_2'$ of~$G_2'$. This means $h \in N_2$, so $\varphi(gn) = h \, \varphi(g) \in N_2 \, \varphi(g)$. Therefore $\varphi(gN_1) \subseteq N_2 \, \varphi(g)$. So $\varphi$ induces a well-defined function $\quot\varphi \colon G_1/N_1 \to G_2/N_2$.
\end{proof}

\begin{defn}[{}{\cite[p.~305]{Li-CISurvey}}]
The \emph{wreath product} (or \emph{lexicographic product}) of two graphs $X_1 = (V_1,E_1)$ and $X_2 = (V_2,E_2)$ is the graph  $X_1[X_2]$ with vertex set $V_1 \times V_2$, such that $(v_1,v_2)$ is adjacent to $(v_1',v_2')$ if and only if either
	\begin{itemize}
	\item $v_1$ is adjacent to $v_1'$ in~$X_1$,
	or
	\item $v_1 = v_1'$ and $v_2$ is adjacent to $v_2'$ in~$X_2$.
	\end{itemize}
\end{defn}

\begin{proof}[Proof of \cref{NilpotentTorsionHaveIso}]
($\Rightarrow$) Let $S_1$ and~$S_2$ be finite, symmetric generating sets of $G_1$ and~$G_2$, respectively, such that there is an isomorphism~$\varphi$ from $\Cay(G_1;S_1)$ to $\Cay(G_2;S_2)$. From \cref{NilpotentTorsion}, we know that $\varphi$ induces a well-defined affine bijection $\quot\varphi \colon G_1/N_1 \to G_2/N_2$. By composing with a left-translation, we may assume $\quot\varphi$ is a group isomorphism. Obviously, this implies $G_1/N_1 \iso G_2/N_2$. Also, since $\quot\varphi$ is a well-defined bijection, we must have $\varphi(N_1) = N_2$. Since $\varphi$ is a bijection, this implies $|N_1| = |N_2|$.

($\Leftarrow$) 
Let 
	\begin{itemize}
	\item $\quot\varphi$ be an isomorphism from $G_1/N_1$ to~$G_2/N_2$, 
	\item $\quot{S_1}$ be a finite generating set of $G_1/N_1$, with $e \notin \quot{S_1}$,
	\item $\quot{S_2} = \quot\varphi(\quot{S_1})$ be the corresponding generating set of $G_2/N_2$,
	and
	\item $S_i = \{\, s \in G_i \mid s N_i \in \quot{S_i} \,\}$, for $i = 1,2$.
	\end{itemize}
Let $n = |N_1| = |N_2|$, and let $E_n$ be the edgeless graph on $n$~vertices. 
Then, for $i = 1,2$, it is easy to see that 
	$\Cay(G_i;S_i)$ is isomorphic to the wreath product $\Cay \bigl( G_i/N_i; \quot{S_i} \bigr) [ E_n ]$.
Since it is obvious that $\quot\varphi$ is an isomorphism from $\Cay \bigl( G_1/N_1; \quot{S_1} \bigr)$ to $\Cay \bigl( G_1/N_2; \quot{S_2} \bigr)$, we have $\Cay(G_1;S_1) \iso \Cay(G_2;S_2)$.
\end{proof}

\begin{proof}[Proof of \cref{UniqueSharp}]
Let $H$ be a sharply transitive, nilpotent subgroup of $\Aut \bigl( \Cay(G;S) \bigr)$. Then a well-known result of G.\,Sabidussi tells us that $\Cay(G;S)$ is isomorphic to a Cayley graph on~$H$ \cite[Prop.~1.1]{Li-CISurvey}, \refnote{Sabidussi}
so \cref{NilpotentTorsionHaveIso} implies $G \iso H$.

From \cref{NilpotentIso}, we see that if $S'$ is any symmetric, finite subset of~$G$, such that $\Cay(G;S')\iso \Cay(G;S)$, then there is a group automorphism~$\alpha$ of~$G$ with $\alpha(S) = S'$.\refnote{AffineToAut} Therefore, since $H$ is a sharply transitive subgroup of $\Aut \bigl( \Cay(G;S) \bigr)$ that is isomorphic to~$G$, a well-known theorem of L.\,Babai tells us that $H$ is  conjugate in $\Aut\bigl(\Cay(G;S)\bigr)$ to the left-regular representation of~$G$ \cite[Thm.~4.1]{Li-CISurvey}.\refnote{Babai} 
However, \cref{NilpotentIsNormal} states that the left-regular representation has no other conjugates in $\Aut \bigl( \Cay(G;S) \bigr)$, so we conclude that $H$ is equal to the left-regular representation of~$G$.
\end{proof}

\section{Other groups that have torsion} \label{OtherTorsionSect}

In this \lcnamecref{OtherTorsionSect}, we prove \Cref{TorsionNotNormal}.
In fact, we prove a more specific version of \cref{TorsionNotNormal}:

\begin{prop}\label{torsion-notnormal}
Suppose $F$ is a nontrivial, finite subgroup of a group~$G$, and $S$ is any finite, symmetric generating set for~$G$. Then $\Cay(G; FSF)$ is a connected Cayley graph of finite valency that is not normal.
\end{prop}

\begin{proof}
It is straightforward to verify that $FSF$ is a symmetric, finite generating set of~$G$, \refnote{FSFGenSet}
so $\Cay(G; FSF)$ is a connected Cayley graph of finite valency.
Furthermore, for all $g \in G$, it is straightforward to check that all vertices in the coset $gF$ have the same neighbours. \refnote{gFTwin} Therefore, if we choose some $h \in gF$ (with $h \neq g$), then there is an automorphism $\varphi$ of $\Cay(G;FSF)$ that interchanges $g$ and~$h$, but fixes all other vertices of the Cayley graph. Since $G$ is infinite, but $FSF$ is finite, we may assume $g$ has been chosen so that $gF$ is disjoint from $FSF \cup \{e\}$. Then $\varphi$ fixes~$e$, but is obviously not a group automorphism, since it fixes every element of the generating set $FSF$, and is not the identity map (since it moves~$g$ to~$h$). So $\varphi$ is not an affine bijection.
\end{proof}

\AtEndDocument{

\newpage

\markboth{Notes to aid the referee}{Notes to aid the referee}

\begin{appendix}

\numberwithin{equation}{aid}

\section{Notes to aid the referee}

\begin{aid} \label{NonIsoGrps}
See \fullcref{TorsionRems}{NonisoGrps} for an example of isomorphic Cayley graphs on non-isomorphic groups.

\begin{defn}
Let $S$ be a subset of a group~$G$.
	\begin{itemize}
	\item $S$ is \emph{symmetric} if it is closed under inverses; that is, $s^{-1} \in S$ for all $s \in S$.
	\item If $S$ is symmetric, then the corresponding \emph{Cayley graph} on~$G$ is the graph $\Cay(G;S)$ whose vertices are the elements of~$G$, and with an edge $g \edge gs$, for all $g \in g$ and $s \in S$.
	\end{itemize}
\end{defn}

\begin{rem}
It is easy to see that $\Cay(G;S)$ is connected if and only if $S$ generates~$G$.
\end{rem}
\end{aid}

\begin{aid} \label{G1isoG2Aid}
We show that \cref{NilpotentIso} implies \cref{G1isoG2}.
Let $\varphi$ be an isomorphism from $\Cay(G_1;S_1)$ to $\Cay(G_2; S_2)$. From \cref{NilpotentIso}, we know there exist a group isomorphism $\alpha \colon G_1 \to G_2$ and $h \in G_2$, such that $\varphi(x) = h \cdot \alpha(x)$, for all $x \in G_1$. Since $\alpha$ is a group isomorphism, we have $G_1 \iso G_2$.
\end{aid}

\begin{aid} \label{regrep}
The \emph{left-regular representation} of~$G$ is the set $\{\, \hat g \mid g \in G\,\}$ of permutations of~$G$, where $\hat g \colon G \to G$ is defined by $\hat g(x) = gx$ for $x \in G$. Since $\widehat{gh} = \hat g \, \hat h$, this is a subgroup of the symmetric group on the set~$G$.
\end{aid}

%

\begin{aid} \label{KleinBottlePf}
It is clear that $\varphi$ is a bijection. The neighbours of $a^i b^j$ are $a^{i \pm 1} b^j$ and $a^i b^{j\pm 1}$. These neighbours are mapped by~$\varphi$ to 
	$$ b^{i \pm 1} a^j = b^i a^j \cdot b^{\pm1} = \varphi(a^i b^j) b^{\pm1}$$
and 
	$$ b^i a^{j \pm 1} = b^i a^j \cdot a^{\pm1} = \varphi(a^i b^j) a^{\pm1} ,$$
which are neighbours of $\varphi(a^i b^j)$. So $\varphi$ is a graph automorphism.

Suppose $\varphi$ is an affine bijection. Since $\varphi(e) = e$ (and $\varphi$ is a bijection), this implies that $\varphi$ is an automorphism of the group~$G$. However, we have $\varphi(a) = b$, and no automorphism of~$G$ can map $a$ to~$b$, since $\langle a \rangle \normal \, G$, but $\langle b \rangle \not\!\normal \,\,\, G$. This is a contradiction.
\end{aid}

\begin{aid} \label{RyabchenkoPf}

\Cref{AbelianCase} follows from the following weaker conclusion that does not require the assumption that $S_i$ generates~$G_i$.

\begin{lem} \label{AbelianCaseSigma}
Assume
	\begin{itemize}
	\item $G_1$ and~$G_2$ are torsion-free, abelian groups,
	\item $S_i$ is a symmetric, finite subset of~$G_i$, for $i = 1,2$,
	and
	\item $\varphi$ is an isomorphism from $\Cay(G_1;S_1)$ to $\Cay(G_2;S_2)$.
	\end{itemize}
Then, for each $g \in G_1$ and $s \in S_1$, there exists $\sigma_g(s) \in S_2$, such that $\varphi(g s^k) = \varphi(g) \, \sigma_g(s)^k$ for all $k \in \ZZ$. 
\end{lem}

\begin{proof}
To simplify the notation, assume $\langle S_i \rangle = G_i$ for $i = 1,2$. (This causes no loss of generality, since $\varphi \bigl( g \langle S_1 \rangle \bigr) = \varphi(g) \, \langle S_2 \rangle$ for all $g \in G$, but a detailed proof works with cosets of $\langle S_i \rangle$, instead of the subgroup $\langle S_i \rangle$ itself.)
\Cref{ConvexGeodesic} provides $s_1 \in S_1$, such that $[s_1^k]_{k=-\infty}^\infty$ is a convex geodesic line in $\Cay(G_1;S_1)$. 
Then $[g s_1^k]_{k=-\infty}^\infty$ is also a convex geodesic line in $\Cay(G_1;S_1)$ (since left-translation is an automorphism of the Cayley graph). Applying the isomorphism~$\varphi$ yields the convex geodesic line $[\varphi(gs_1^k)]_{k=-\infty}^\infty$ in $\Cay(G_2;S_2)$. Now \cref{EdgeInCentre} implies that all edges in this geodesic line have the same label (since $G_2$ is abelian). This means there is some $\sigma_g(s_1) \in S_2$, such that $\varphi(g s_1^k) = \varphi(g) \, \sigma_g(s_1)^k$ for all $k \in \ZZ$. This is the desired conclusion for $s = s_1$.

Now, we make the important observation that if $\sigma_g(s)$ exists, for some $s \in S_1$, then $\sigma_g(s) = \sigma_h(s)$ for all $g,h \in \langle S_1 \rangle$. Namely, for all $k \in \ZZ$, we have
	\begin{align*}
	\dist_{S_1}(g , h) 
	&= \dist_{S_1}(g s^k, h s^k) 
	&& \text{($G_1$ is abelian)}
	\\&= \dist_{S_2} \bigl( \varphi(g s^k), \varphi(h s^k) \bigr)
	&& \text{($\varphi$ is an isomorphism)}
	\\&= \dist_{S_2} \bigl( \varphi(g) \, \sigma_g(s)^k, \varphi(h) \, \sigma_h(s)^k\bigr)
	,\end{align*}
so \fullcref{NilpLems}{BddDist} tells us that $\sigma_g(s) = \sigma_h(s)$. 

Therefore, $\sigma_g(s_1)$ is a constant (since we assumed at the start of the proof that $\langle S_1 \rangle = G_1$; without this assumption, it would only be constant on cosets of $\langle S_1 \rangle$). Calling this constant~$s_2$ yields $\varphi(gs_1) = \varphi(g) \, s_2$ for all $g \in G_1$. 
Letting $S_i' = S_i \smallsetminus \{s_i^{\pm1}\}$ for $i = 1,2$, this implies that $\varphi$ is an isomorphism from $\Cay(G_1;S_1')$ to $\Cay(G_2;S_2')$. 
By induction on the valency, we conclude that the desired $\sigma_g(s)$ exists for all $s \in S_1 \smallsetminus \{s_1^{\pm1}\}$. Since the first paragraph provides $\sigma_g(s_1)$, this completes the proof.
\end{proof}

\begin{proof}[Proof of \cref{AbelianCase}]
Since $\langle S_1 \rangle = G_1$, the second paragraph of the proof of the \lcnamecref{AbelianCaseSigma} tells us that $\sigma_g(s) = \sigma_h(s)$ for all $g,h \in G_1$, so we may drop the subscript: $\varphi(gs) = \varphi(g) \, \sigma(s)$ for all $g \in G_1$ and $s \in S_1$.
Since $S_1$ generates~$G_1$, and $\varphi$~is a bijection, this implies that $\varphi$ is an affine bijection. 
\end{proof}
\end{aid}


\begin{aid} \label{IsolateNormal}
Let $g \in G$ and $x \in \sqrt{N}$. There  is some $k > 0$ with $x^k \in N$. Since $N \normaleq G$, we have 
	$$ (g^{-1} x g)^k = g^{-1} x^k g \in g^{-1} N g = N ,$$
so $g^{-1} x g \in \sqrt{N}$. Therefore $\sqrt{N} \normaleq G$.

Suppose $g \sqrt{N}$ is a torsion element of $G/\sqrt{N}$. This means there is some $k \neq 0$ with $\left( g \sqrt{N} \right)^k = \sqrt{N}$, so $g^k \in \sqrt{N}$. This means there is some $\ell \neq 0$ with $(g^k)^\ell \in N$. Therefore $g^{k\ell} \in N$ (and $k \ell \neq 0$), so $g \in \sqrt{N}$. Therefore $g \sqrt{N}$ is trivial. So $G/\sqrt{N}$ is torsion-free.
\end{aid}

\begin{aid} \label{rankH<rankG}

\begin{lem}[{}{\cite[Lem.~2.6, p.~9]{Hall-NilpGrps}}] \label{subnormal}
Let $G$ be nilpotent of class~$c$ and let $H$ be a proper subgroup of~$G$. Define $H_0 = H$ and, inductively, $H_{i+1}$ to be the normalizer of~$H_i$ in~$G$. Then
	$$ H = H_0 < H_1 < \cdots < H_r = G $$
for some $r \le c$.
\end{lem}

\begin{proof}[Proof of \fullcref{NilpLems}{rank})]
Let $H_i$ be as in \cref{subnormal}. Since $H_{i+1}$ is the normalizer of~$H_i$, we may write
	$$ H = H_0 \normal H_1 \normal \cdots \normal H_r = G .$$
From \fullcref{NilpLems}{RankQuotient} and induction, we have
	$$ \rank G = \rank H + \sum_{i = 1}^r \rank(H_i/H_{i-1}). $$
So $\rank H \le \rank G$, with equality if and only if $\rank(H_i/H_{i-1}) = 0$ for all~$i$. 

Since it is clear that $\rank F = 0$ if and only if $F$ is finite, this means that $\rank H = \rank G$ if and only if $H_i/H_{i-1}$ is finite for all~$i$. This is the case if and only if $G/H$ is finite.
\end{proof}
\end{aid}

\begin{aid} \label{FinitelyManyConjugates}
If we take the special case of $\pi$-isolated where $\pi$ is the set of all prime numbers, \cite[2.3.8(i), p.~42]{LennoxRobinson} says:
	\begin{itemize}
	\item[] Suppose $H$ is a subgroup of a torsion-free, nilpotent group~$G$. 
	\item[] Then $C_G(H)$ is isolated for every subgroup~$H$.
	\end{itemize}
To say that $C_G(H)$ is ``isolated"  means that if $g^k \in C_G(H)$ for some nonzero $k \in \ZZ$, then $g \in C_G(H)$ \cite[first paragraph of \S2.3, p.~38]{LennoxRobinson}.

Now, suppose $h$~has only finitely many conjugates. This means $C_G(h)$ is a finite-index subgroup of~$G$, so there is some nonzero $k \in \ZZ$, such that $g^k \in C_G(h)$ for all $g \in G$. From the preceding paragraph, we conclude that $g \in C_G(h)$. Since $g$ is an arbitrary element of~$G$, this means $h \in Z(G)$.
\end{aid}

\begin{aid} \label{BddDist}
Since $\dist_{S}(a^k, gb^k)$ is bounded as a function of~$k$, we know that 
	$$ \text{$\{\, a^{-k} g b^k \mid k \in \ZZ \,\}$ is finite} .$$
Hence, there exist $k \neq \ell$, such that $a^{-k} g b^k = a^{-\ell} g \, b^\ell$, so, letting $m = \ell - k \neq 0$, we have 
$g^{-1} a^m g = b^m$. In other words, $(g^{-1} a g)^m = b^m$. Since $G$ is torsion-free nilpotent, this implies $g^{-1} a g = b$ \cite[2.1.2, p.~30]{LennoxRobinson}.
\end{aid}

\begin{aid} \label{DistortedNote}
The paper \cite{Conner-TransNumbs} uses the following notation:
	\begin{itemize}
	\item \cite[Defns.~2.2 and 2.3]{Conner-TransNumbs} $\|x\| = \dist_S(e,x)$ (this is called a ``word metric'')
	\item \cite[Lem.~2.43(i)]{Conner-TransNumbs} $\tau(x) = \lim_{n \to \infty} \|x^n\|/n$
	\item \cite[Defn.~2.5]{Conner-TransNumbs} $I(G) = \{\, g \in G \mid \tau(g) = 0 \,\}$
	\item \cite[Defn.~3.1]{Conner-TransNumbs} $B(G) = \{\, g \in G \mid \tau(gx) = \tau(x), \ \forall x \in G \,\}$.
	\item \cite[Notn.~3.2(ii)]{Conner-TransNumbs} $G' = [G,G]$
	\end{itemize}

\begin{lem}[{}{\cite[Lem.~3.5(i,iii)]{Conner-TransNumbs}}] \label{TransNumLem}
Let $G$ be a nilpotent group. Then
\begin{enumerate}
\item[(i)] $B(G) = I(G)$
\item[(iii)] If $G$ is finitely generated and equipped with a word metric then $B(G) = \sqrt{G'}$.
\end{enumerate}
\end{lem}

\begin{proof}[Proof of \fullcref{NilpLems}{distorted}]
Translating to the notation of \cite{Conner-TransNumbs}, we have 
	$$ \dist_S(e,g^k)/k \rightarrow 0 
	\Leftrightarrow  \|g\\^k\|/k \rightarrow 0 
	\Leftrightarrow \tau(g) = 0 
	\Leftrightarrow g \in I(G)
	.$$
From \cref{TransNumLem}, we have $I(G) = B(G) = \sqrt{G'} = \sqrt{[G,G]}$.
\end{proof}
\end{aid}

\begin{aid} \label{InverseOrder}
We have
	$$ x \prec' y 
	\Rightarrow x^{-1} \prec y^{-1} 
	\Rightarrow b^{-1} x^{-1} a^{-1} \prec b^{-1} y^{-1} a^{-1}
	\Rightarrow axb \prec' ayb ,$$
so $\prec'$ is invariant under both left-translations and right-translations.

Also, from the definition of~$\prec'$, we have $s \succeq (S \cup S^{-1}) \Leftrightarrow s^{-1} \succeq' (S \cup S^{-1})$.
\end{aid}

\begin{aid} \label{ac<bd}
Since $a \preceq b$, invariance under right-translations implies $ac \preceq bc$ (with equality iff $a = b$). Since $c \preceq d$, invariance under left-translations implies $bc \preceq bd$ (with equality iff $c = d$). Now transitivity implies $ac \preceq bd$ (with equality iff $a = b$ and $c = d$).

For the base case of a proof by induction, note that the maximality of~$s$ implies $s_1 \preceq s$ (with equality iff $s_1 = s$). Now suppose $s_1 s_2 \cdots s_k \preceq s^k$ (with equality iff $s_1 = s_2 = \cdots = s_k = s$). Since $s_1 s_2 \cdots s_k \preceq s^k$ and $s_{k+1} \preceq s$, we have 
	$$s_1 s_2 \cdots s_{k+1} = s_1 s_2 \cdots s_k \cdot s_{k+1} \preceq s^k \cdot s = s^{k+1} ,$$
with equality iff $s_1 s_2 \cdots s_k = s^k$ and $s_{k+1} = s$. However, we have already noted that $s_1 s_2 \cdots s_k = s^k$ implies $s_1 = s_2 = \cdots = s_k = s$.
\end{aid}

\begin{aid} \label{TranslateToIdentity}
Let $h = \varphi(e)$, and define $\varphi'(x) = h^{-1} \cdot \varphi(x)$. Then $\varphi'$ is an isomorphism from $\Cay(G_1;S_1)$ to $\Cay(G_2;S_2)$ with $\varphi'(e) = e$. If $\varphi'$ is an affine bijection, then $\varphi$ is also an affine bijection.
\end{aid}

\begin{aid} \label{g=e}
Let $h = \varphi(g)$, and define $\varphi'(x) = h^{-1} \cdot \varphi(g x)$. Then $\varphi'$ is an isomorphism from $\Cay(G_1;S_1)$ to $\Cay(G_2;S_2)$ with $\varphi'(e) = e$. If there is some $g' \in G_2$, such that $\varphi'(z^k) = \varphi'(e) \, (g')^k$, for all $k \in \ZZ$, then
	$$ \varphi(gz^k) = h \cdot \varphi'(z^k) = \varphi(g) \cdot \varphi'(e) \, (g')^k = \varphi(g) \, (g')^k ,$$
so we may let $\sigma_g(z) = g'$.
\end{aid}

\begin{aid} \label{psi(z)=s}
The definition of~$S_2^*$ provides $\rho \in \Aut_e \bigl( \Cay(G_2;S_2) \bigr)$ with $\rho(g_2) = s$. Since $g_2 = \varphi(z)$, we may let $\psi$ be the composition $\rho \circ \varphi$.
\end{aid}

\begin{aid} \label{CommutatorSmall}
Suppose $G_1^*$ has finite index in~$G_1$. Then $\sqrt{G_1^*} = G_1$, so  \fullcref{NilpLems}{CommutatorIsolated} implies
	$$ [G_1,G_1] = [\sqrt{G_1^*}, \sqrt{G_1^*}] \subseteq \sqrt{[G_1^*,G_1^*]} , $$
so $[G_1^*,G_1^*]$ has finite index in $[G_1,G_1]$ \fullcsee{NilpLems}{FiniteIndex}. This is a contradiction.
\end{aid}

\begin{aid} \label{Zconnected}
For a graph~$\Gamma$ and $r \in \ZZ^+$, the $r$th power of~$\Gamma$ is the graph $\Gamma^r$ with the same vertex set as~$\Gamma$, and with an edge from $u$ to~$v$ iff $\dist_\Gamma(u,v) \le r$. It is clear that:
	\begin{itemize}
	\item Any isomorphism from $\Gamma_1$ to~$\Gamma_2$ is also an isomorphism from~$\Gamma_1^r$ to~$\Gamma_2^r$.
	\item $\Cay(G;S)^r = \Cay(G; S^r)$, where $S^r$ is the set of all elements of~$G$ that can be written as a product of $\le r$ elements of~$S$.
	\end{itemize}
Since $Z_1^\dagger$ is finitely generated, it has a finite generating set. For any sufficiently large~$r$, this finite set is contained in $S_1^r$. Since $\varphi$ is an isomorphism from $\Cay(G_1;S_1^r)$ to $\Cay(G_2;S_2^r)$, there is no harm in replacing $S_1$ and~$S_2$ with $S_1^r$ and~$S_2^r$.
\end{aid}

\begin{aid} \label{G/ZdaggerTorsFree}
We have
	\begin{align*} 
	\sqrt{Z_i^\dagger}
	&= \sqrt{Z(G_i) \cap \sqrt{[G_i,G_i]}}
	&& (\text{definition of $Z_i^\dagger$})
	\\&= \sqrt{Z(G_i)} \cap \sqrt{\sqrt{[G_i,G_i]}}
	&& (\sqrt{H \cap K} = \sqrt{H} \cap \sqrt{K})
	\\&= Z(G_i) \cap \sqrt{[G_i,G_i]}
	&& \text{(\fullcref{NilpLems}{Zisolated} and $\sqrt{\sqrt{H}} = \sqrt{H}$)}
	\\&= Z_i^\dagger 
	&& (\text{definition of $Z_i^\dagger$})
	, \end{align*}
so $G/Z_i^\dagger$ is torsion-free.
\end{aid}

\begin{aid} \label{Nstar}
Suppose we can show that the result is true for $\langle N^*\rangle$.
Let $f$ be an automorphism of $\Cay(G;S)$ that fixes $e$, and let $f^*$ be the restriction of $f$ to $\langle N^*\rangle$.  Since $N^*$ is invariant, we know that $f^*$ is an automorphism of $\Cay(\langle N^*\rangle; N^*)$.  Also, it is clear from the definition of $N^*$ that $N$ is contained in $N^*$.  (Also, $\langle N^*\rangle/N$ is bi-orderable, because it is a subgroup of $G/N$.)  Therefore, if we know the theorem is true for $\langle N^*\rangle$, then $f^*(N)$ is contained in $N$. Since $f^*(N) = f(N)$, this means that $f(N)$ is contained in $N$, as desired.
\end{aid}

\begin{aid} \label{Sabidussi}
\begin{prop}[Sabidussi, 1964]
A graph~$\Gamma$ is isomorphic to a Cayley graph on a group~$G$ if and only if $\Aut \Gamma$ contains a sharply transitive subgroup that is isomorphic to~$G$.
\end{prop}

\medskip

Now, let $N$ be the torsion subgroup of~$H$.
Since $G$ and~$H$ both have a Cayley graph isomorphic to $\Cay(G;S)$ (and the torsion subgroup of~$G$ is trivial), \cref{NilpotentTorsionHaveIso} tells us that $G/\{e\} \iso H/N$ and $|\{e\}| = |N|$. So $G \iso H$.
\end{aid}

\begin{aid} \label{AffineToAut}
Let $\varphi$ be an isomorphism from $\Cay(G;S)$ to $\Cay(G;S')$. From \cref{NilpotentIso}, we know that $\varphi$ is an affine bijection, so there exist a group automorphism $\alpha$ of~$G$ and $h \in G$, such that $\varphi(x) = h \cdot \alpha(x)$ for all $x \in G$. Since $\varphi$ is a graph isomorphism, we have $\varphi(xS) = \varphi(x) S'$ for all $x \in S$. Taking $x = e$, this yields
	$$ h \cdot \alpha(S) = \varphi(eS) = \varphi(e) S' = h \cdot \alpha(e) S' = h \cdot S' ,$$
so $\alpha(S) = S'$.
\end{aid}

\begin{aid} \label{Babai}
The following result is traditionally stated only for finite groups, but the same proof works in general.

\begin{prop}[Babai, 1977]
For a group $G$, the following two conditions are equivalent:
\begin{itemize}
\item whenever $S$ and $S'$ are finite, symmetric generating sets of $G$ and $\Cay(G;S) \cong \Cay(G;S')$, there is an automorphism $\alpha$ of $G$ with $\alpha(S)=S'$; 
\item for every finite, symmetric generating set~$S$ of~$G$, the left-regular representation of~$G$ is conjugate to every subgroup of $\Aut \bigl( \Cay(G;S) \bigr)$ that is isomorphic to~$G$ and acts sharply transitively on the vertices of $\Cay(G;S)$. 
\end{itemize}
\end{prop}
\end{aid}

\begin{aid} \label{FSFGenSet}
We have $(FSF)^{-1} = F^{-1} S^{-1} F^{-1} = FSF$ (since $F$ and~$S$ are symmetric), so $FSF$ is symmetric.
Also, it is clear that $F S F$ is finite, since $F$ and~$S$ are both finite. Finally, since $e \in F$ (because $F$ is a subgroup), we have $S = e \cdot S \cdot e \subseteq F S F$, so $F S F$ generates~$G$. 
\end{aid} 

\begin{aid} \label{gFTwin}
For $f \in F$, the set of neighbours of $gf$ is $gf \cdot FSF = g \cdot (fF) \cdot SF = gFSF$, which is the set of neighbours of~$g$.
\end{aid}

\end{appendix}

}


\begin{thebibliography}{99}

\bibitem{Conner-TransNumbs}
{\scshape Conner, G.R.}
Properties of translation numbers in nilpotent groups.
{\em Comm. Algebra} 26 (1998) 1069--1080. 
\mrev{1612184} (99d:20051),
\zbl{0902.20015}

\bibitem{Godsil-FullAut}
{\scshape Godsil, C.D.}
On the full automorphism group of a graph.
{\em Combinatorica} {\bf 1} (1981) 243--256. 
\mrev{0637829} (83a:05066),
\zbl{0489.05028}

\bibitem{Gromov-PolyGrowth}
{\scshape Gromov, M.}
Groups of polynomial growth and expanding maps.
{\em Inst. Hautes \'Etudes Sci. Publ. Math.} {\bf 53} (1981) 53--73. 
\mrev{0623534} (83b:53041),
\zbl{0474.20018}

\bibitem{Hall-NilpGrps}
{\scshape Hall, P.}
The Edmonton notes on nilpotent groups.
{\em Queen Mary College Mathematics Notes. 
Mathematics Department, Queen Mary College, London,} 1969. 
\mrev{0283083} (44~\#316),
\zbl{0211.34201}

\bibitem{KopytovMedvedev}
{\scshape Kopytov, V.M.; Medvedev, N.Ya.}
Right-ordered groups.
{\em Consultants Bureau, New York,} 1996. 
\mrev{1393199} (97h:06024b),
\zbl{0896.06017}

\bibitem{LennoxRobinson}
{\scshape Lennox, J.C.; Robinson, D.J.S.}
The theory of infinite soluble groups.
{\em Oxford University Press, Oxford,} 2004. 
\mrev{2093872} (2006b:20047),
\zbl{1059.20001}

\bibitem{Li-CISurvey}
{\scshape Li, C.H.}
On isomorphisms of finite Cayley graphs---a survey.
{\em Discrete Math.} {\bf 256} (2002) 301--334. 
\mrev{1927074} (2003i:05067),
\zbl{1018.05044}

\bibitem{Loeh-WhichAbel}
{\scshape L\"oh, C.}
Which finitely generated Abelian groups admit isomorphic Cayley graphs?
{\em Geom. Dedicata} {\bf 164} (2013) 97--111. 
\mrev{3054618},
\zbl{1266.05059}

\bibitem{MoellerSeifter}
{\scshape M\"oller, R.G.; Seifter, N.}
Digraphical regular representations of infinite finitely generated groups.
{\em European J. Combin.} {\bf 19} (1998) 597--602. 
\mrev{1637768} (99i:20007),
\zbl{0905.05036}


\bibitem{Morris-CIInfinite-Arxiv}
{\scshape Morris, J.}
The CI problem for infinite groups.
Preprint.
\url{http://arxiv.org/abs/1502.06114}

\bibitem{Robinson-CourseGroups}
{\scshape Robinson, D.J.S.}
A course in the theory of groups, 2nd ed.
{\em Springer-Verlag, New York,} 1996.
\mrev{1357169 (96f:20001)},
\zbl{0836.20001}

\bibitem{Ryabchenko-CIAbelian}
{\scshape Ryabchenko, A.A.}
Isomorphisms of Cayley graphs of a free abelian group (Russian).
{\em Sibirsk. Mat. Zh.} {\bf 48} (2007) 1142--1146; 
English translation in
{\em Siberian Math. J.} {\bf 48} (2007) 919--922.
\mrev{2364633} (2008i:05085),
\zbl{1164.05382}

\bibitem{MathOverflow-IsomCayley}
{\scshape Sharf, M.}
Isometries of some simple Cayley graphs.
\emph{MathOverflow} \#200836,
\url{http://mathoverflow.net/q/200836}.

\end{thebibliography}
\end{document}